\documentclass[twoside,a4paper,10pt]{amsart}

\usepackage{fullpage}
\usepackage{amscd}
\usepackage{amsmath}
\usepackage{amssymb}
\usepackage{amsthm}
\usepackage{color}
\usepackage{hyperref}
\usepackage{lmodern}
\usepackage{mathrsfs}
\usepackage{mathtools}
\usepackage{pict2e}
\usepackage{stackrel}
\usepackage{tikz}
\usepackage{wasysym}
\usepackage[all]{xy}
\usepackage{xcolor}

\hypersetup{
    colorlinks,
    linkcolor={red!50!black},
    citecolor={blue!50!black},
    urlcolor={blue!80!black}
}

\usepackage{palatino}
\allowdisplaybreaks

\newtheorem{lemma}{Lemma}[section]
\newtheorem{remark}[lemma]{Remark}

\newtheorem{theorem}[lemma]{Theorem}
\newtheorem{corollary}[lemma]{Corollary}
\newtheorem{definition}[lemma]{Definition}
\newtheorem{proposition}[lemma]{Proposition}

\author{Daniel Robert-Nicoud}
\date{}
\title{Representing the deformation $\infty$-groupoid}
\address{Laboratoire Analyse, G\'eom\'etrie et Applications, Universit\'e Paris 13, Sorbonne Paris Cit\'e, 99 Avenue Jean Baptiste Cl\'ement, 93430 Villetaneuse, France}
\email{daniel.robertnicoud@gmail.com}

\renewcommand{\bar}{\ensuremath{\mathrm{B}}}

\newcommand{\C}{\ensuremath{\mathscr{C}}}

\newcommand{\com}{\ensuremath{\mathit{Com}}}

\newcommand{\del}{\ensuremath{\mathrm{Del}}}

\newcommand{\g}{\ensuremath{\mathfrak{g}}}

\newcommand{\h}{\ensuremath{\mathfrak{h}}}
\newcommand{\id}{\ensuremath{\mathrm{id}}}
\newcommand{\lie}{\ensuremath{\mathit{Lie}}}

\newcommand{\mc}{\ensuremath{\mathfrak{mc}}}
\newcommand{\MC}{\ensuremath{\mathrm{MC}}}

\renewcommand{\k}{\ensuremath{\mathbb{K}}}
\renewcommand{\L}{\ensuremath{\mathscr{L}}}
\renewcommand{\P}{\ensuremath{\mathscr{P}}}

\newcommand{\Q}{\ensuremath{\mathscr{Q}}}

\newcommand{\rect}{R_\bullet}

\renewcommand{\S}{\ensuremath{\mathbb{S}}}

\newcommand{\susp}{\ensuremath{\mathscr{S}}}

\newcommand{\sSet}{\ensuremath{\mathsf{sSet}}}

\makeatletter
\newcommand{\adjunction}{\@ifstar\named@adjunction\normal@adjunction}
\newcommand{\normal@adjunction}[4]{%
  #1\colon #2%
  \mathrel{\vcenter{%
    \offinterlineskip\m@th
    \ialign{%
      \hfil$##$\hfil\cr
      \longrightharpoonup\cr
      \noalign{\kern-.3ex}
      \smallbot\cr
      \longleftharpoondown\cr
    }%
  }}%
  #3 \noloc #4%
}
\newcommand{\named@adjunction}[4]{%
  #2%
  \mathrel{\vcenter{%
    \offinterlineskip\m@th
    \ialign{%
      \hfil$##$\hfil\cr
      \scriptstyle#1\cr
      \noalign{\kern.1ex}
      \longrightharpoonup\cr
      \noalign{\kern-.3ex}
      \smallbot\cr
      \longleftharpoondown\cr
      \scriptstyle#4\cr
    }%
  }}%
  #3%
}
\newcommand{\longrightharpoonup}{\relbar\joinrel\rightharpoonup}
\newcommand{\longleftharpoondown}{\leftharpoondown\joinrel\relbar}
\newcommand\noloc{%
  \nobreak
  \mspace{6mu plus 1mu}
  {:}
  \nonscript\mkern-\thinmuskip
  \mathpunct{}
  \mspace{2mu}
}
\newcommand{\smallbot}{%
  \begingroup\setlength\unitlength{.15em}%
  \begin{picture}(1,1)
  \roundcap
  \polyline(0,0)(1,0)
  \polyline(0.5,0)(0.5,1)
  \end{picture}%
  \endgroup
}
\makeatother



\thanks{The author was supported by grants from R\'egion Ile-de-France, and the grant ANR-14-CE25-0008-01 project SAT}

\begin{document}

\begin{abstract}
	The goal of the present paper is to introduce a smaller, but equivalent version of the deformation $\infty$-groupoid associated to a homotopy Lie algebra. In the case of differential graded Lie algebras, we represent it by a universal cosimplicial object.
\end{abstract}

\maketitle

\setcounter{tocdepth}{1}

\tableofcontents

\section{Introduction}

The \emph{fundamental principle of deformation theory}, due to Deligne, Grothendieck and many others and recently formalized and proved in the context of $\infty$-categories by Pridham and Lurie, states that

\medskip

\begin{center}
	\begin{minipage}{0.8\textwidth}
		\begin{center}
			``Every deformation problem in characteristic $0$ is encoded in the space of Maurer--Cartan elements of a differential graded Lie algebra.''
		\end{center}
	\end{minipage}
\end{center}

\medskip

Therefore, one is naturally led to the study of Maurer--Cartan elements of differential graded Lie algebras and, more generally, homotopy Lie algebras.

\medskip

In order to encode the Maurer--Cartan elements, gauge equivalences between them, and higher relations between gauge equivalences, Hinich \cite{hinich97} introduced the Deligne--Hinich $\infty$-groupoid. It is a Kan complex associated to any complete $\L_\infty$-algebra modeling the space of its Maurer--Cartan elements. Since it is a very big object, Getzler introduced in \cite{getzler09} a smaller but weakly equivalent Kan complex $\gamma_\bullet$ which, however, is more difficult to manipulate. In this paper, we introduce another simplicial set associated to any $\L_\infty$-algebra of which we prove the following nice properties:
\begin{enumerate}
	\item it is weakly equivalent to the Deligne--Hinich $\infty$-groupoid,
	\item it is a Kan complex,
	\item it is contained in the Getzler $\infty$-groupoid $\gamma_\bullet$, and
	\item if we restrict to the category of complete dg Lie algebras, there is an explicit cosimplicial dg Lie algebra $\mc_\bullet$ representing this object.
\end{enumerate}
The cosimplicial dg Lie algebra $\mc_\bullet$ was already introduced in the works of Buijs--Murillio--F\'elix--Tanr\'e \cite{buijs15} in the context of rational homotopy theory. We show here that it plays a key role in deformation theory.

\medskip

Results coming from operad theory play a crucial role throughout the paper, especially in the second part. In particular, we use the explicit formul{\ae} for the $\infty$-morphisms induced by the Homotopy Transfer Theorem given in \cite{vallette12} and various theorems proven in \cite{rn17tensor}.

\medskip

Little after the apparition of the present article, Buijs--Murillio--F\'elix--Tanr\'e gave an alternative proof of Corollary \ref{cor:exMainThm} in \cite{buijs17}. Their proof doesn't rely on general operadic results, but rather on explicit combinatorial computations.

\medskip

The author was made aware by Marco Manetti in a private conversation that many of the results of this article are already present in the unpublished PhD thesis \cite{bandiera} of his student Ruggero Bandiera (now also appeared in the article \cite{bandiera17}). We acknowledge this, but we consider that the present article remains interesting in that the methods used to prove the results are different. In particular, in view of Bandiera's results, Sections \ref{sect:mainThm} and \ref{sect:properties} can be interpreted as an alternative construction of the Getzler $\infty$-groupoid $\gamma_\bullet$ with new proofs of its properties.

\subsection*{Structure of the paper}

In Section \ref{sect:DHGgpd} we give a short review of the Deligne groupoid, the Deligne--Hinich $\infty$-groupoid, and the main theorems in this context. In Section \ref{sect:mainThm} we state and prove our main theorem, giving a new simplicial set encoding the Maurer--Cartan space of $\L_\infty$-algebras. Next, in Section \ref{sect:properties}, we study some properties of this object. In particular, we prove that it is a Kan complex, and that it is ``small'' in a precise sense. Finally, we focus on the special case of dg Lie algebras in Section \ref{sect:LieAlgebras}, showing that our Kan complex is represented by a cosimplicial dg Lie algebra in this situation.

\subsection*{Notation and conventions}

We work over a fixed field $\k$ of characteristic $0$.

\medskip

We abbreviate ``differential graded'' by dg, and sometimes omit it completely. All algebras are differential graded unless stated otherwise.

\medskip

Since we work with differential forms, we adopt the cohomological convention. Therefore, we work over cochain complexes, and Maurer--Cartan elements of dg Lie and $\L_\infty$-algebras (i.e. homotopy Lie algebras) are in degree $1$, not $-1$. All cochain complexes are $\mathbb{Z}$-graded.

\medskip

We use the letter $s$ to denote a formal element of degree $1$. If $C_\bullet$ is a cochain complex, then $sC_\bullet$ denotes the suspension of $C_\bullet$, which is sometimes written as $C_\bullet[1]$.

\medskip

We sometimes denote the identity maps by $1$.

\medskip

By a filtered $\L_\infty$-algebra we mean a pair $(\g,F_\bullet\g)$ where $\g$ is an $\L_\infty$-algebra and $F_\bullet\g$ is a descending filtration of $\g$ such that $F_1\g = \g$ and
\begin{enumerate}
	\item for all $n\ge1$, we have $d_\g(F_n\g)\subseteq F_n\g$,
	\item for all $k\ge2$ and $n_1,\ldots,n_k\ge1$ we have
	\[
	\ell_k(F_{n_1}\g,\ldots,F_{n_k}\g)\subseteq F_{n_1+\cdots+n_k}\g\ ,
	\]
	and
	\item the $\L_\infty$-algebra $\g$ is complete with respect to the filtration, i.e.
	\[
	\g\cong\varprojlim_n\g/F_n\g
	\]
	as $\L_\infty$-algebras.
\end{enumerate}
When the context is clear, we write $\g^{(n)}\coloneqq\g/F_n\g$. For details about (filtered) $\L_\infty$-algebras and the definitions and basic properties about (filtered) $\infty$-morphisms we refer the reader to the article \cite{dolgushev15}.

\subsection*{Acknowledgments}

I thank Ezra Getzler, Marco Manetti, Chris Rogers and Jim Stasheff for their useful comments, both editorial and mathematical. I am naturally also extremely grateful to my advisor Bruno Vallette for all the help, support and continuous discussion. I am thankful to Yangon, Don Mueang, and Narita international Airports for the almost pleasant working atmosphere they provided, as the bulk of this paper was typed in between flights.

\section{The deformation \texorpdfstring{$\infty$}{oo}-groupoid} \label{sect:DHGgpd}

An object of fundamental interest in deformation theory is the \emph{Deligne groupoid} $\del(\g)$ associated to a complete dg Lie algebra $\g$. There is a higher generalization of the Deligne groupoid in the form of the Deligne--Hinich $\infty$-groupoid. It is a simplicial set with nice properties and whose $1$-truncation gives back the Deligne groupoid. It was introduced in \cite{hinich97} and then studied in depth and further generalized in \cite{getzler09}.

\subsection{The Deligne groupoid}

Let $\g$ be a dg Lie algebra. Then we can associate a groupoid $\del(\g)$ to $\g$, called the \emph{Deligne groupoid}, as follows. The objects of the Deligne groupoid are the \emph{Maurer--Cartan elements} of $\g$, i.e. the degree $1$ elements $\alpha\in\g^1$ satisfying the Maurer--Cartan equation
\[
d\alpha + \frac{1}{2}[\alpha,\alpha] = 0\ .
\]

\begin{definition}
	The set of Maurer--Cartan elements of $\g$ is denoted by $\MC(\g)$.
\end{definition}

We have the set of objects of $\del(\g)$, we still need to define its morphisms. To an element $\lambda\in\g^0$, one can associate a ``vector field'' by sending $\alpha\in\g^1$ to
\[
d\lambda + [\lambda,\alpha]\in\g^1.
\]
It is tangent to the Maurer--Cartan locus, in the sense that if $\alpha(t)$ is the flow of $\lambda$, that is
\[
\dot{\alpha}(t) = d\lambda + [\lambda,\alpha(t)]
\]
with $\alpha(0)\in\MC(\g)$, then $\alpha(t)\in\MC(\g)$ for all $t$, whenever it exists. We say that two Maurer--Cartan elements $\alpha_0,\alpha_1\in\MC(\g)$ are \emph{gauge equivalent} if there exists such a flow $\alpha(t)$ such that $\alpha(i) = \alpha_i$ for $i=0,1$. The Deligne groupoid is the groupoid associated to this equivalence relation, which means that the morphisms are
\[
\del(\g)(\alpha_0,\alpha_1) \coloneqq \{\lambda\in\g^0\mid\text{the flow of }\lambda\text{ starting at }\alpha_0\text{ gives }\alpha_1\text{ at time }1\}\ .
\]
For further reference, see for example \cite{goldman88}.

\medskip

The assignment of the Deligne groupoid to a dg Lie algebra is functorial and has a good homotopical behavior: it sends filtered quasi-isomorphisms to equivalences, as can be seen by the Goldman-Millson theorem, which was first proven in \cite{goldman88}, and then generalized e.g. in \cite{yekutieli12}.

\subsection{Generalization: the deformation $\infty$-groupoid}

Let $\g$ be a nilpotent $\L_\infty$-algebra. The Maurer--Cartan equation can be generalized to
\[
dx + \sum_{n\ge2}\frac{1}{n!}\ell_n(x,\ldots,x) = 0
\]
for $x\in\g^1$. Again, we denote by $\MC(\g)$ the set of all elements satisfying this equation.

\begin{remark}
	Notice that the condition that $\g$ be nilpotent is sufficient to make it so that the left-hand side of the Maurer--Cartan equation is well defined.
\end{remark}

\subsubsection{The Deligne--Hinich $\infty$-groupoid}

\begin{definition}
	The \emph{Sullivan algebra} is the simplicial dg commutative algebra
	\[
	\Omega_n\coloneqq \k[t_0,\ldots,t_n,dt_0,\ldots,dt_n]/\left(\sum_{i=0}^nt_i = 1,\sum_{i=0}^ndt_i = 0\right)
	\]
	with $|t_i|=0$ and endowed with the unique differential satisfying $d(t_i) = dt_i$.
\end{definition}

This object was introduced by Sullivan in the context of rational homotopy theory \cite{sullivan77}. At level $n$, it is the algebra of polynomial differential forms on the standard geometric $n$-simplex. Now let $\g$ be a nilpotent $\L_\infty$-algebra. Then tensoring $\g$ with $\Omega_n$ gives us back a nilpotent $\L_\infty$-algebra, of which we can consider the Maurer--Cartan elements.

\begin{definition}
	The \emph{Deligne--Hinich $\infty$-groupoid} is the simplicial set
	\[
	\MC_\bullet(\g)\coloneqq\MC(\g\otimes\Omega_\bullet)\ .
	\]
	This association is natural in $\g$, and thus defines a functor
	\[
	\MC_\bullet:\{\text{nilpotent }\L_\infty\text{-algebras}\}\longrightarrow\sSet\ .
	\]
\end{definition}

We will rather consider the following slight generalization: Let $(\g,F_\bullet\g)$ be a filtered $\L_\infty$-algebra, then
\[
\g\cong\varprojlim_n\g/F_n\g
\]
is the limit of a sequence of nilpotent $\L_\infty$-algebras. Thus we can define
\[
\MC_\bullet(\g)\coloneqq\varprojlim_n\MC_\bullet(\g/F_n\g)\ .
\]
Notice that the elements in $\MC_\bullet(\g)$ in this case are not polynomials with coefficients in $\g$ anymore, but rather power series with some ``vanishing at infinity'' conditions. We state all the following results in this setting.

\begin{theorem} \label{thm:Hinich}
	Let either:
	\begin{itemize}
		\item \cite[Prop. 4.7]{getzler09}: $\g,\h$ be nilpotent $\L_\infty$-algebras and $\Phi:\g\to\h$ be a surjective strict morphism of $\L_\infty$-algebras, or
		\item \cite[Thm. 2]{rogers16}: $\g,\h$ be filtered $\L_\infty$-algebras and $\Phi:\g\rightsquigarrow\h$ be a filtered $\infty$-morphism that induces a surjection at every level of the filtrations.
	\end{itemize}
	 Then
	\[
	\MC_\bullet(\Phi):\MC_\bullet(\g)\longrightarrow\MC_\bullet(\h)
	\]
	is a fibration of simplicial sets. In particular, for any filtered $\L_\infty$-algebra $\g$, the simplicial set $\MC_\bullet(\g)$ is a Kan complex.
\end{theorem}

This result was originally proven by Hinich \cite[Th. 2.2.3]{hinich97} for strict surjections between nilpotent dg Lie algebras concentrated in positive degrees, and then generalized by E. Getzler and by C. L. Rogers to the version stated above.

\medskip

Generalizing the Goldman--Millson theorem, V. A. Dolgushev and C. L. Rogers proved in \cite[Thm. 2.2]{dolgushev15} that the Deligne--Hinich $\infty$-groupoid behaves well with respect to homotopy theory: it sends filtered quasi-isomorphisms of filtered $\L_\infty$-algebras to weak equivalences.

\subsubsection{Basic forms, Dupont's contraction and Getzler's functor $\gamma_\bullet$}

The Sullivan algebra has a subcomplex $C_\bullet$ linearly spanned by the \emph{basic forms}
\[
\omega_I\coloneqq k!\sum_{j=1}^k(-1)^jt_{i_j}dt_{i_0}\ldots\widehat{dt_{i_j}}\ldots dt_{i_k}\in\Omega_n
\]
for $I=\{i_0<i_1<\cdots<i_k\}\subseteq\{0,\ldots,n\}$. This is in fact the (co)cellular complex for the standard geometric $n$-simplex $\Delta^n$. In order to prove a simplicial version of the de Rham theorem, J. L. Dupont \cite{dupont76} introduced a homotopy retraction
\begin{center}
	\begin{tikzpicture}
		\node (a) at (0,0){$\Omega_\bullet$};
		\node (b) at (2,0){$C_\bullet$};
		
		\draw[->] (a)++(.3,.1)--node[above]{\mbox{\tiny{$p_\bullet$}}}+(1.4,0);
		\draw[<-,yshift=-1mm] (a)++(.3,-.1)--node[below]{\mbox{\tiny{$i_\bullet$}}}+(1.4,0);
		\draw[->] (a) to [out=-150,in=150,looseness=4] node[left]{\mbox{\tiny{$h_\bullet$}}} (a);
	\end{tikzpicture}
\end{center}
where all the maps are simplicial. Homotopy retraction means that we have
\[
p_\bullet i_\bullet = 1,\qquad\mathrm{and}\qquad 1-i_\bullet p_\bullet = dh_\bullet + h_\bullet d\ .
\]
Moreover, the maps satisfy the \emph{side conditions}
\[
h_\bullet i_\bullet = 0,\qquad p_\bullet h_\bullet = 0,\qquad\text{and}\qquad h_\bullet^2 = 0\ .
\]
A homotopy retraction satisfying the side conditions is called a \emph{contraction}.

\medskip

This contraction will be a fundamental ingredient in the rest of the present paper. As the Deligne--Hinich $\infty$-groupoid is always a big object, Getzler defined the following subobject.

\begin{definition}
	The \emph{Getzler $\infty$-groupoid} is the sub-simplicial set $\gamma_\bullet(\g)$ of $\MC_\bullet(\g)$ given by
	\[
	\gamma_n(\g) \coloneqq \{\alpha\in\MC_n(\g)\mid h_n\alpha = 0\}\ .
	\]
\end{definition}

\begin{theorem}[\cite{getzler09}]
	The simplicial set $\gamma_\bullet(\g)$ is a Kan complex, and it is weakly equivalent to the Deligne--Hinich $\infty$-groupoid $\MC_\bullet(\g)$.
\end{theorem}

A part of the definition of $h_\bullet$ and $p_\bullet$ which we will need in what follows is the (formal) integration of a form in the Sullivan algebra over a simplex, which is given by:
\[
\int_{\Delta^n}t_1^{a_1}\ldots t_n^{a_n}dt_1\ldots dt_n \coloneqq \frac{a_1!\cdots a_n!}{(a_1+\cdots+a_n+n)!}\ .
\]
It corresponds to the usual integration when working over $\k=\mathbb{R}$.

\begin{remark}
	We have
	\[
	\int_{\Delta^p}\omega_I = 1
	\]
	for $p + 1=|I|$, where $\Delta^p$ is the subsimplex of $\Delta^n$ with vertices indexed by $I$.
\end{remark}

\begin{definition}
	A form $\alpha\in\gamma_n(\g)$ is said to be \emph{thin} if
	\[
	\int_{\Delta^n}\alpha = 0\ .
	\]
\end{definition}

\begin{theorem}[\cite{getzler09}]
	For every horn in $\gamma_\bullet(\g)$, there exists a unique thin simplex filling it.
\end{theorem}

\begin{remark}
	The existence of a set of thin simplices such that every horn has a unique thin filler is what is meant by Getzler when he speaks of an $\infty$-groupoid. We use the term simply to mean Kan complex (e.g. when speaking of the Deligne--Hinich $\infty$-groupoid).
\end{remark}

\section{Main Theorem} \label{sect:mainThm}

In this section, we give a reminder on the Homotopy Transfer Theorem for commutative and for $\L_\infty$-algebras, before going on to state and prove the main theorem of the present article.

\subsection{Reminder on the Homotopy Transfer Theorem}

Let $V,W$ be  cochain complexes, and suppose that we have a retraction
\begin{center}
	\begin{tikzpicture}
		\node (a) at (0,0){$V$};
		\node (b) at (2,0){$W\ ,$};
		
		\draw[->] (a)++(.3,.1)--node[above]{\mbox{\tiny{$p$}}}+(1.4,0);
		\draw[<-,yshift=-1mm] (a)++(.3,-.1)--node[below]{\mbox{\tiny{$i$}}}+(1.4,0);
		\draw[->] (a) to [out=-150,in=150,looseness=4] node[left]{\mbox{\tiny{$h$}}} (a);
	\end{tikzpicture}
\end{center}
that is, we have
\[
ip - 1 = dh + hd
\]
and $pi = 1$. Furthermore, we can always suppose that
\[
h^2 = 0,\qquad hi = 0,\qquad\text{and}\qquad ph = 0\ ,
\]
see e.g. \cite[p. 365]{ls87}. The Homotopy Transfer Theorem tells us that we can coherently transfer algebraic structures from $V$ to $W$. More precisely, the specific cases of interest to us are the following ones.

\begin{theorem}[Homotopy Transfer Theorem for commutative algebras]
	Suppose $V$ is a commutative algebra. There is a $\C_\infty$-algebra structure on $W$ such that $p$ and $i$ extend to $\infty$-quasi isomorphisms $p_\infty$ and $i_\infty$ of $\C_\infty$-algebras between $V$ and $W$ endowed with the respective structures.
\end{theorem}

\begin{theorem}[Homotopy Transfer Theorem for $\L_\infty$-algebras]
	Suppose $V$ is an $\L_\infty$-algebra. There is an $\L_\infty$-algebra structure on $W$ such that $p$ and $i$ extend to $\infty$-quasi isomorphisms $p_\infty$ and $i_\infty$ of $\L_\infty$-algebras between $V$ and $W$ endowed with the respective structures.
\end{theorem}

For details on this theorem, see e.g. \cite[Sect. 10.3]{vallette12}, where it is proven in the general context of algebras over operads. See also \cite[Sect. 10.3.5--6]{vallette12} for the explicit formul{\ae} for the $\infty$-morphisms $p_\infty$ and $i_\infty$.

\subsection{Statement of the main theorem}

Let $\g$ be a complete $\L_\infty$-algebra. The Dupont contraction induces a contraction
\begin{center}
	\begin{tikzpicture}
		\node (a) at (0,0){$\g\otimes\Omega_\bullet$};
		\node (b) at (2.8,0){$\g\otimes C_\bullet$};
		
		\draw[->] (a)++(.6,.1)--node[above]{\mbox{\tiny{$1\otimes p_\bullet$}}}+(1.6,0);
		\draw[<-,yshift=-1mm] (a)++(.6,-.1)--node[below]{\mbox{\tiny{$1\otimes i_\bullet$}}}+(1.6,0);
		\draw[->] (a) to [out=-160,in=160,looseness=4] node[left]{\mbox{\tiny{$1\otimes h_\bullet$}}} (a);
	\end{tikzpicture}
\end{center}
of $\g\otimes\Omega_\bullet$ to $\g\otimes C_\bullet$. Applying the Homotopy Transfer Theorem to this contraction, we obtain a simplicial $\L_\infty$-algebra structure on $\g\otimes C_\bullet$. We also know that we can extend the maps $1\otimes p_\bullet$ and $1\otimes i_\bullet$ to simplicial $\infty$-morphisms of simplicial $\L_\infty$-algebras $(1\otimes p_\bullet)_\infty$ and $(1\otimes i_\bullet)_\infty$. Notice that these $\infty$-morphisms are indeed simplicial because they are given by sums of compositions of copies of $1\otimes i_\bullet$, $1\otimes p_\bullet$, $1\otimes h_\bullet$, and the brackets of $\g\otimes\Omega_\bullet$, all of which respect the simplicial structure. We denote $P_\bullet$ and $I_\bullet$ the induced maps on Maurer--Cartan elements. We will also use the notation
\[
(1\otimes r_\bullet)_\infty \coloneqq (1\otimes i_\bullet)_\infty(1\otimes p_\bullet)_\infty\ ,
\]
and we dub $\rect$ the map induced by $(1\otimes r_\bullet)_\infty$ on Maurer--Cartan elements.

\begin{theorem} \label{thm:mainThm}
	Let $\g$ be a filtered $\L_\infty$-algebra. The maps $P_\bullet$ and $I_\bullet$ are inverse one to the other in homotopy, and thus provide a weak equivalence
	\[
	\MC_\bullet(\g)\simeq\MC(\g\otimes C_\bullet)
	\]
	of simplicial sets which is natural in $\g$.
\end{theorem}

\begin{remark}
	The simplicial $\L_\infty$-algebra $\g\otimes C_\bullet$ has the advantage of being quite smaller than $\g\otimes\Omega_\bullet$, since $C_n$ is finite dimensional for each $n$. The price to pay is that the algebraic structure is much more convoluted.
\end{remark}

\subsection{Proof of the main theorem}

The rest of this section is dedicated to the proof of this result. We begin with the following lemma.

\begin{lemma} \label{lemma:PIis1}
	We have
	\[
	P_\bullet I_\bullet = \id_{\MC(\g\otimes C_\bullet)}\ .
	\]
\end{lemma}

\begin{proof}
	This is because $(1\otimes p_\bullet)_\infty(1\otimes i_\bullet)_\infty$ is the identity, see e.g. \cite[Theorem 5]{dotsenko16}, and the functoriality of the Maurer--Cartan functor $\MC$.
\end{proof}

Therefore, it is enough to prove that the map
\[
\rect = I_\bullet P_\bullet:\MC_\bullet(\g)\longrightarrow\MC_\bullet(\g)
\]
is a weak equivalence. The idea is to use the same methods as in \cite{dolgushev15}. The situation is however slightly different, as the map $\rect$ is not of the form $\Phi\otimes1_{\Omega_\bullet}$, and thus Theorem 2.2 of loc. cit. cannot be directly applied. The first, easy step is to understand what happens at the level of the zeroth homotopy group.

\begin{lemma} \label{lemma:pi0}
	The map
	\[
	\pi_0(\rect):\pi_0\MC_\bullet(\g)\longrightarrow\pi_0\MC_\bullet(\g)
	\]
	is a bijection.
\end{lemma}

\begin{proof}
	We have $\Omega_0=C_0=\k$, and the maps $i_0$ and $p_0$ both are the identity of $\k$. Therefore, the map $R_0$ is the identity of $\MC_0(\g)$, and thus obviously induces a bijection on $\pi_0$.
\end{proof}

For the higher homotopy groups, we start with a simplified version of \cite[Prop. 2.4]{dolgushev15}, which gives in some sense the base for an inductive argument. If the $\L_\infty$-algebra $\g$ is abelian, i.e. all of its brackets vanish, then so do the brackets at all levels of $\g\otimes\Omega_\bullet$. In this case, the Maurer--Cartan elements are exactly the cocycles of the underlying cochain complex, and therefore $\MC_\bullet(\g)$ is a simplicial vector space.

\begin{lemma} \label{lemma:abelianDGL}
	If the $\L_\infty$-algebra $\g$ is abelian, then $\rect$ is a weak equivalence of simplicial vector spaces.
\end{lemma}

\begin{proof}
	Recall that the Moore complex of a simplicial vector space $V_\bullet$ is defined by
	\[
	\mathcal{M}(V_\bullet)_n\coloneqq s^nV_n
	\]
	endowed with the differential
	\[
	\partial\coloneqq\sum_{i=0}^n(-1)^id_i\ ,
	\]
	where the maps $d_i$ are the face maps of the simplicial set $V_\bullet$. It is a standard result that
	\[
	\pi_0(V_\bullet) = H_0(\mathcal{M}(V_\bullet))\ ,\qquad \pi_i(V_\bullet,v)\cong\pi_i(V_\bullet,0) = H_i(\mathcal{M}(V_\bullet))
	\]
	for all $i\ge1$ and $v\in V_0$, and that a map of simplicial vector spaces is a weak equivalence if and only if it induces a quasi-isomorphism between the respective Moore complexes \cite[Cor. 2.5, Sect. III.2]{goerss09}.
	
	In our case,
	\[
	V_\bullet\coloneqq\MC_\bullet(\g) = \mathcal{Z}^1(\g\otimes\Omega_\bullet)
	\]
	is the simplicial vector space of $1$-cocycles of $\g\otimes\Omega_\bullet$. As in \cite{dolgushev15}, it can be proven that the map
	\[
	\mathcal{M}(1\otimes p_\bullet):\mathcal{M}(\mathcal{Z}^1(\g\otimes\Omega_\bullet))\longrightarrow\mathcal{M}(\mathcal{Z}^1(\g\otimes C_\bullet))
	\]
	is a quasi-isomorphism.  But as the bracket vanishes, this is exactly $P_\bullet$. Now
	\[
	\mathcal{M}(1\otimes p_\bullet)\mathcal{M}(1\otimes i_\bullet) = 1_{\mathcal{M}(\mathcal{Z}^1(\g\otimes\Omega_\bullet))}\ ,
	\]
	which implies that $\mathcal{M}(1\otimes i_\bullet)$ also is a quasi-isomorphism. It follows that $\rect$ is a weak equivalence, concluding the proof.
\end{proof}

Now we basically follow the structure of \cite[Sect. 4]{dolgushev15}. We define a filtration of $\g\otimes\Omega_\bullet$ by
\[
F_k(\g\otimes\Omega_\bullet)\coloneqq (F_k\g)\otimes\Omega_\bullet\ .
\]
We denote by
\[
(\g\otimes\Omega_\bullet)^{(k)}\coloneqq\g\otimes\Omega_\bullet/F_k(\g\otimes\Omega_\bullet) = \g^{(k)}\otimes\Omega_\bullet\ .
\]
The composite $(1\otimes i_\bullet)(1\otimes p_\bullet)$ induces an endomorphism $(1\otimes i_\bullet)^{(k)}(1\otimes p_\bullet)^{(k)}$ of $(\g\otimes\Omega_\bullet)^{(k)}$. All the $\infty$-morphisms coming into play obviously respect this filtration, and moreover $1\otimes h_\bullet$ passes to the quotients, so that we have
\[
1_{(\g\otimes\Omega_\bullet)^{(k)}} - (1\otimes i_\bullet)^{(k)}(1\otimes p_\bullet)^{(k)} = d(1\otimes h_\bullet)^{(k)} + (1\otimes h_\bullet)^{(k)}d
\]
for all $k$, which shows that $(1\otimes r_\bullet)_\infty$ is a filtered $\infty$-quasi isomorphism.

\medskip

The next step is to reduce the study of the homotopy groups with arbitrary basepoint to the study of the homotopy groups with basepoint $0\in\MC_0(\g)$.

\begin{lemma} \label{lemma:shiftToZero}
	Let $\alpha\in\MC(\g)$, and let $\g^\alpha$ be the $\L_\infty$-algebra obtained by twisting $\g$ by $\alpha$, that is the $\L_\infty$-algebra with the same underlying graded vector space, but with differential
	\[
	d^\alpha(x) \coloneqq dx + \sum_{n\ge2}\frac{1}{(n-1)!}\ell_n(\alpha,\ldots,\alpha,x)
	\]
	and brackets
	\[
	\ell^\alpha(x_1,\ldots,x_m) \coloneqq \sum_{n\ge m}\frac{1}{(n-m)!}\ell_n(\alpha,\ldots,\alpha,x_1,\ldots,x_m)\ .
	\]
	Let
	\[
	\mathrm{Shift}_\alpha:\MC_\bullet(\g^\alpha)\longrightarrow\MC_\bullet(\g)
	\]
	be the isomorphism of simplicial sets induced by the map given by
	\[
	\beta\in\g\longmapsto\alpha+\beta\in\g^\alpha\ .
	\]
	Then the following diagram commutes
	\[
	\begin{CD}
		\MC_\bullet(\g^\alpha) @>\mathrm{Shift}_\alpha>> \MC_\bullet(\g)\\
		@V{\rect^\alpha}VV @VV{\rect}V\\
		\MC_\bullet(\g^{\alpha}) @>\mathrm{Shift}_\alpha>> \MC_\bullet(\g)\\
	\end{CD}
	\]
	where
	\[
	\rect^\alpha(\beta)\coloneqq \sum_{k\ge1}(1\otimes r_\bullet)^\alpha_k(\beta^{\otimes k})
	\]
	and
	\[
	(1\otimes r_\bullet)^\alpha_k(\beta_1\otimes\cdots\otimes\beta_k)\coloneqq\sum_{j\ge0}\frac{1}{j!}(1\otimes r_\bullet)_{k+j}(\alpha^{\otimes j}\otimes\beta_1\otimes\ldots\otimes\beta_k)
	\]
	is the twist of $(1\otimes r_\bullet)_\infty$ by the Maurer--Cartan element $\alpha$. Here, we identified $\alpha\in\g$ with $\alpha\otimes1\in\g\otimes\Omega_\bullet$.
\end{lemma}

\begin{proof}
	The proof in \cite[Lemma 4.3]{dolgushev16} goes through \emph{mutatis mutandis}.
\end{proof}

\begin{remark}
	The $\L_\infty$-algebra $\g^\alpha$ in Lemma \ref{lemma:shiftToZero} is endowed with the same filtration as $\g$.
\end{remark}

Now we proceed by induction to show that $R^{(k)}$ is a weak equivalence from $\MC_\bullet(\g^{(k)})$ to itself for all $k\ge2$. As the $\L_\infty$-algebra $(\g\otimes\Omega_\bullet)^{(2)}$ is abelian, the base step of the induction is given by Lemma \ref{lemma:abelianDGL}.

\begin{lemma} \label{lemma:inductionStep}
	Let $m\ge2$. Suppose that
	\[
	\rect^{(k)}:\MC(\g^{(k)})\longrightarrow\MC(\g^{(k)})
	\]
	is a weak equivalence for all $2\le k\le m$. Then $\rect^{(m+1)}$ is also a weak equivalence.
\end{lemma}

\begin{proof}
	The zeroth homotopy set $\pi_0$ has already been taken care of in Lemma \ref{lemma:pi0}. Thanks to Lemma \ref{lemma:shiftToZero}, it is enough to prove that $\rect^{(m+1)}$ induces isomorphisms of homotopy groups $\pi_i$ based at $0$, for all $i\ge1$.
	
	Consider the following commutative diagram
	\[
	\begin{CD}
		0 @>>> \frac{F_m(\g\otimes\Omega_\bullet)}{F_{m+1}(\g\otimes\Omega_\bullet)} @>>> (\g\otimes\Omega_\bullet)^{(m+1)} @>>> (\g\otimes\Omega_\bullet)^{(m)} @>>> 0\\
		@. @VVV @V{(1\otimes r_\bullet)_\infty^{(m+1)}}VV @VV{(1\otimes r_\bullet)_\infty^{(m)}}V\\
		0 @>>> \frac{F_m(\g\otimes\Omega_\bullet)}{F_{m+1}(\g\otimes\Omega_\bullet)} @>>> (\g\otimes\Omega_\bullet)^{(m+1)} @>>> (\g\otimes\Omega_\bullet)^{(m)} @>>> 0
	\end{CD}
	\]
	where the leftmost vertical arrow is given by the linear term $(1\otimes i_\bullet)(1\otimes p_\bullet)$ of $(1\otimes r_\bullet)_\infty$ since all higher terms vanish, as can be seen by the explicit formul{\ae} for the $\infty$-quasi isomorphisms induced by the Homotopy Transfer Theorem given in \cite[Sect. 10.3.5--6]{vallette12}. Therefore, it is a weak equivalence as the $\L_\infty$-algebras in question are abelian. The first term in each row is the fiber of the next map, which is surjective. By Theorem \ref{thm:Hinich}, we know that applying the $\MC$ functor makes the horizontal maps on the right into fibrations of simplicial sets, while the objects we obtain on the left are easily seen to be the fibers Taking the long sequence in homotopy and using the five-lemma, we see that all we are left to do is to prove that $\rect^{(m+1)}$ induces an isomorphism on $\pi_1$. Notice that it is necessary to prove this, as the long sequence is exact everywhere except on the level of $\pi_0$.
	
	\medskip
	
	The long exact sequence of homotopy groups (truncated on both sides) reads
	\[
	\pi_2\MC_\bullet(\g^{(m)})\stackrel{\partial}{\longrightarrow}\pi_1\MC_\bullet\left(\frac{F_m\g}{F_{m+1}\g}\right)\longrightarrow\pi_1\MC_\bullet(\g^{(m+1)})\longrightarrow\pi_1\MC_\bullet(\g^{(m)})\stackrel{\partial}{\longrightarrow}\pi_0\MC_\bullet\left(\frac{F_m\g}{F_{m+1}\g}\right)\ ,
	\]
	where in the higher homotopy groups we left the basepoint implicit (as it is always $0$). The map
	\[
	\partial:\pi_1\MC_\bullet(\g^{(m)})\longrightarrow\pi_0\MC_\bullet\left(\frac{F_m\g}{F_{m+1}\g}\right) = H^1(F_{m+1}\g/F_m\g)
	\]
	is seen to be the obstruction to lifting an element of $\pi_1\MC_\bullet(\g^{(m)})$ to an element of $\pi_1\MC_\bullet(\g^{(m+1)})$ (e.g. \cite[Lemma 7.3]{goerss09}).
	
	\medskip
	
	\textit{The map $\pi_1(\rect^{(m+1)})$ is surjective:} Let $y\in\pi_1\MC_\bullet(\g^{(m+1)})$ and denote by $\overline{y}$ its image in $\pi_1\MC_\bullet(\g^{(m)})$. By the induction hypothesis, there exists a unique $\overline{x}\in\pi_1\MC_\bullet(\g^{(m)})$ which is mapped to $\overline{y}$ under $\rect^{(m)}$. As $\overline{y}$ is the image of $y$, we have $\partial(\overline{y}) = 0$, and this implies that $\partial(\overline{x})=0$, too. Therefore, there exists $x\in\pi_1\MC_\bullet(\g^{(m+1)})$ mapping to $\overline{x}$. Denote by $y'$ the image of $x$ under $\rect^{(m+1)}$. Then $y'y^{-1}$ is in the kernel of the map
	\[
	\pi_1\MC_\bullet(\g^{(m+1)})\longrightarrow\pi_1\MC_\bullet(\g^{(m)})\ .
	\]
	By exactness of the long sequence, and the fact that $\rect$ induces an automorphism of	$\pi_1\MC_\bullet\left(\frac{F_{m+1}\g}{F_m\g}\right)$, there exists an element $z\in\pi_1(\MC_\bullet(F_{m+1}\g/F_m\g))$ mapping to $y'y^{-1}$ under the composite
	\[
	\pi_1\MC_\bullet\left(\frac{F_{m+1}\g}{F_m\g}\right)\stackrel{\rect}{\longrightarrow}\pi_1\MC_\bullet\left(\frac{F_{m+1}\g}{F_m\g}\right)\longrightarrow\pi_1\MC_\bullet(\g^{(m+1)})\ .
	\]
	Let $x'$ be the image of $z$ in $\pi_1\MC_\bullet(\g^{(m+1)})$, then $(x')^{-1}x$ maps to $y$ under $\rect^{(m+1)}$. This proves the surjectivity of the map $\pi_1(\rect^{(m+1)})$.

	\medskip
	
	\textit{The map $\pi_1(\rect^{(m+1)})$ is injective:} Assume $x,x'\in\pi_1\MC_\bullet(\g^{(m+1)})$ map to the same element under $\rect^{(m+1)}$. Then $x(x')^{-1}$ maps to the neutral element $0$ under $\rect^{(m+1)}$. It follows that there is a $z\in\pi_1\MC_\bullet\left(\frac{F_{m+1}\g}{F_m\g}\right)$ mapping to $x(x')^{-1}$ which must be such that its image $w$ is itself the image of some $\tilde{w}\in\pi_2\MC_\bullet(\g^{(m)})$ under the map $\partial$. But by the induction hypothesis and the exactness of the long sequence, this implies that $z$ is in the kernel of the next map, and thus that $x(x')^{-1}$ is the identity element. Therefore, the map $\pi_1(\rect^{(m+1)})$ is injective.

	\medskip

	This ends the proof of the lemma.
\end{proof}

Finally, we can conclude the proof of Theorem \ref{thm:mainThm}.

\begin{proof}[Proof of Theorem \ref{thm:mainThm}]
	Lemma \ref{lemma:inductionStep}, together with all we have said before, shows that $\rect^{(m)}$ is a weak equivalence for all $m\ge2$. Therefore, we have the following commutative diagram:
	\[
	\begin{CD}
		\vdots @. \vdots\\
		@VVV @VVV\\
		\MC_\bullet(\g^{(4)}) @>{\sim}>> \MC_\bullet(\g^{(4)})\\
		@VVV @VVV\\
		\MC_\bullet(\g^{(3)}) @>{\sim}>> \MC_\bullet(\g^{(3)})\\
		@VVV @VVV\\
		\MC_\bullet(\g^{(2)}) @>{\sim}>> \MC_\bullet(\g^{(2)})\\
	\end{CD}
	\]
	where all objects are Kan complexes, all horizontal arrows are weak equivalences, and all vertical arrows are (Kan) fibrations by Theorem \ref{thm:Hinich}. It follows that the collection of horizontal arrows defines a weak equivalence between fibrant objects in the model category of towers of simplicial sets, see \cite[Sect. VI.1]{goerss09}. The functor from towers of simplicial sets to simplicial sets given by taking the limit is right adjoint to the constant tower functor, which trivially preserves cofibrations and weak equivalences. Thus, the constant tower functor is a left Quillen functor, and it follows that the limit functor is a right Quillen functor. In particular, it preserves weak equivalences between fibrant objects. Applying this to the diagram above proves that $\rect$ is a weak equivalence.
\end{proof}

\begin{remark}
	As an anonymous referee pointed out, there is an alternative, shorter proof of the fact that the map $\rect$ induces a bijection on all higher homotopy groups: A. Berglund \cite[Thm. 1.1]{ber15} gave an explicit group isomorphism
	\[
	\mathrm{B}:H_n(\g)\longrightarrow\pi_{n+1}\MC_\bullet(\g)\ ,\qquad n\ge0\ ,
	\]
	for any complete $\L_\infty$-algebra $\g$. One can use this map together with the explicit formula for the map $\rect$ derived from the Homotopy Transfer Theorem to immediately derive the result.
	
	\medskip
	
	In \cite{rnv18} an alternative proof of Berglund's theorem is given which relies on the results of the present article. It is therefore important to have a demonstration of Theorem \ref{thm:mainThm} which does not depend on it.
\end{remark}

\section{Properties and comparison} \label{sect:properties}

Theorem \ref{thm:mainThm} shows that the simplicial set $\MC(\g\otimes C_\bullet)$ is a new model for the Deligne--Hinich $\infty$-groupoid. This section is dedicated to the study of some properties of this object. We start by showing that it is a Kan complex, then we give some conditions on the differential forms representing its simplices. We show how we can use it to rectify cells of the Deligne--Hinich $\infty$-groupoid, which provides an alternative, simpler proof of \cite[Lemma B.2]{dolgushev15}. Finally we compare it with Getzler's functor $\gamma_\bullet$, proving that our model is contained in Getzler's. Independent results by Bandiera \cite{bandiera}, \cite{bandiera17} imply that the two models are actually isomorphic.

\subsection{Properties of $\MC_\bullet(\g\otimes C_\bullet)$}

The following proposition is the analogue to Theorem \ref{thm:Hinich} for our model.

\begin{proposition}
	Let $\g,\h$ be two filtered $\L_\infty$-algebras, and suppose that $\phi:\g\to\h$ is a morphism of $\L_\infty$-algebras inducing a fibration of simplicial sets under the functor $\MC_\bullet$, see for example Theorem \ref{thm:Hinich} for possible sufficient conditions. Then the induced morphism
	\[
	\MC(\phi\otimes\id_{C_\bullet}):\MC(\g\otimes C_\bullet)\longrightarrow\MC(\h\otimes C_\bullet)
	\]
	is also a fibration of simplicial sets. In particular, for any filtered $\L_\infty$-algebra $\g$, the simplicial set $\MC(\g\otimes C_\bullet)$ is a Kan complex.
\end{proposition}

\begin{proof}
	By assumption, the morphism
	\[
	\MC_\bullet(\phi):\MC_\bullet(\g)\longrightarrow\MC_\bullet(\h)
	\]
	is a fibration of simplicial set, and by Lemma \ref{lemma:PIis1} the following diagram exhibits $\MC(\phi\otimes\id_{C_\bullet})$ as a retract of $\MC_\bullet(\phi)$.
	\begin{center}
		\begin{tikzpicture}
			\node (a) at (0,0) {$\MC(\g\otimes C_\bullet)$};
			\node (b) at (3,0) {$\MC_\bullet(\g)$};
			\node (c) at (6,0) {$\MC(\g\otimes C_\bullet)$};
			\node (d) at (0,-2) {$\MC(\g\otimes C_\bullet)$};
			\node (e) at (3,-2) {$\MC_\bullet(\g)$};
			\node (f) at (6,-2) {$\MC(\g\otimes C_\bullet)$};
			
			\draw[->] (a) -- node[above]{$I_\bullet$} (b);
			\draw[->] (b) -- node[above]{$P_\bullet$} (c);
			\draw[->] (d) -- node[below]{$I_\bullet$} (e);
			\draw[->] (e) -- node[below]{$P_\bullet$} (f);
			\draw[->] (a) -- node[left]{$\MC(\phi\otimes\id_{C_\bullet})$} (d);
			\draw[->] (b) -- node[left]{$\MC_\bullet(\phi)$} (e);
			\draw[->] (c) -- node[right]{$\MC(\phi\otimes\id_{C_\bullet})$} (f);
		\end{tikzpicture}
	\end{center}
	As the class of fibrations is closed under retracts, this concludes the proof.
\end{proof}

We consider the composite $\rect=I_\bullet P_\bullet$, which is not the identity.

\begin{definition}
	We call the morphism
	\[
	\rect:\MC_\bullet(\g)\longrightarrow\MC_\bullet(\g)
	\]
	the \emph{rectification map}.
\end{definition}

The following result is a wide generalization of \cite[Lemma B.2]{dolgushev15}, as well as a motivation for the name ``rectification map'' for $\rect$.

\begin{proposition} \label{prop:forElModel}
	We consider an element
	\[
	\alpha \coloneqq \alpha_1(t_0,\ldots,t_n) + \cdots\in\MC_n(\g)\ ,
	\]
	where the dots indicate terms in $\g^{1-k}\otimes\Omega_n^k$ with $1\le k\le n$. Then $\beta\coloneqq \rect(\alpha)\in\MC_n(\g)$ is of the form
	\[
	\beta = \beta_1(t_0,\ldots,t_n) + \cdots + \xi \otimes\omega_{0\ldots n}\ ,
	\]
	where the dots indicate terms in $\g^{1-k}\otimes\Omega_n^k$ with $1\le k\le n-1$, where $\xi$ is an element of $\g^{1-n}$, and where $\alpha_1$ and $\beta_1$ agree on the vertices of $\Delta^n$. In particular, if $\alpha\in\MC_1(\g)$, then $\beta = F(\alpha)\in\MC_1(\g)$ is of the form
	\[
	\beta = \beta_1(t) +\lambda dt
	\]
	for some $\lambda\in\g^0$, and satisfies
	\[
	\beta_1(0) = \alpha_1(0)\qquad\text{and}\qquad\beta_1(1) = \alpha_1(1)\ ,
	\]
	so that $\lambda$ gives a gauge equivalence between $\alpha_1(0)$ and $\alpha_1(1)$.
\end{proposition}

\begin{remark}
	As $\rect$ is a projector, this proposition in fact gives information on the form of all the elements of $\MC(\g\otimes C_\bullet)$.
\end{remark}

\begin{proof}
	First notice that the map $\rect$ commutes with the face maps and is the identity on $0$-simplices, thus evaluation of the part of $\beta$ in $\g^1\otimes\Omega_n^0$ at the vertices gives the same result as evaluation at the vertices of $\alpha_1$.	Next, we notice that $\beta$ is in the image of $I_\bullet$. We use the explicit formula for $(1\otimes i_n)_\infty$ of \cite[Sect. 10.3.5]{vallette12}: the operator acting on arity $k\ge2$ is given, up to signs, by the sum over all rooted trees with $1\otimes i_n$ put at the leaves, the brackets $\ell_n$ of the corresponding arity at all vertices, and $1\otimes h$ at the inner edges and at the root. But the $1\otimes h$ at the root lowers the degree of the part of the form in $\Omega_n$ by $1$, and thus we cannot get something in $\g^{1-n}\otimes\Omega_n^n$ from these terms. The only surviving term is therefore the one coming from $(1\otimes i_n)(P_\bullet(\alpha))$, given by $\xi\otimes\omega_{0\ldots n}$ for some $\xi\in\g^{1-n}$.
\end{proof}

\subsection{Comparison with Getzler's $\infty$-groupoid $\gamma_\bullet$}

Finally, we compare the simplicial set $\MC(\g\otimes C_\bullet)$ with Getzler's Kan complex $\gamma_\bullet(\g)$. We start with an easy result that follows directly from our approach, before exposing Bandiera's result that these two simplicial sets are actually isomorphic.

\begin{lemma} \label{lemma:comparisonGamma}
	We have
	\[
	I_\bullet\MC(\g\otimes C_\bullet)\subseteq\gamma_\bullet(\g)\ .
	\]
\end{lemma}

\begin{proof}
	We have $h_\bullet i_\bullet = 0$. Therefore, by the explicit formula for $(i_\bullet)_\infty$ given in \cite[Sect. 10.3.5]{vallette12}, we have $h_\bullet(\beta) = 0$ for any $\beta\in\g\otimes\Omega_\bullet$ in the image of $I_\bullet$. Thus
	\[
	h_\bullet(\MC(\g\otimes C_\bullet)) = h_\bullet I_\bullet P_\bullet(\MC_\bullet(\g)) = 0\ ,
	\]
	which proves the claim.
\end{proof}

In his thesis \cite{bandiera}, Bandiera proves the following.

\begin{theorem}[{\cite[Thm. 2.3.3 and Prop 5.2.7]{bandiera}}] \label{thm:bandiera}
	The map
	\[
	(P_\bullet,1\otimes h_\bullet):\MC_\bullet(\g)\longrightarrow\MC(\g\otimes C_\bullet)\times\big(\mathrm{Im}(1\otimes h_\bullet)\cap (\g\otimes\Omega_\bullet)^1\big)
	\]
	is bijective. In particular, its restriction to $\gamma_\bullet(\g) = \ker(1\otimes h_\bullet)\cap\MC_\bullet(\g)$ gives an isomorphism of simplicial sets
	\[
	P_\bullet:\gamma_\bullet(\g)\longrightarrow\MC(\g\otimes C_\bullet)\ .
	\]
\end{theorem}

\begin{remark}
	Thanks to our approach, we immediately have an inverse for the map $P_\bullet$: it is of course the map $I_\bullet$.
\end{remark}

As a consequence of Bandiera's result and of Proposition \ref{prop:forElModel}, we can partially characterize the thin elements of $\gamma_\bullet(\g)$.

\begin{lemma}
	For each $n\ge1$, the thin elements contained in $\gamma_n(\g)$ are those with no term in $\g^{1-n}\otimes\Omega_n^n$.
\end{lemma}

\begin{proof}
	By Proposition \ref{prop:forElModel} and Theorem \ref{thm:bandiera}, we know that if $\alpha\in\gamma_n(\g)$, then $\alpha$ is of the form
	\[
	\alpha = \cdots + \xi\otimes\omega_{0\ldots n}
	\]
	for some $\xi\in\g^{1-n}$, where the dots indicate terms in $\g^{1-k}\otimes\Omega_n^k$ for $0\le k\le n-1$, which will give zero after integration. Integrating, we get
	\[
	\int_{\Delta^n}\alpha = \xi\otimes\int_{\Delta^n}\omega_{0\ldots n} = \xi\otimes1\ .
	\]
	Therefore, $\alpha$ is thin if, and only if $\xi=0$. 
\end{proof}

\section{The case of Lie algebras} \label{sect:LieAlgebras}

In this section, we focus on the case where $\g$ is actually a dg Lie algebra. In this situation, we are able to represent the functor $\MC(\g\otimes C_\bullet)$ by a cosimplicial dg Lie algebra. The main tools used here are results from the article \cite{rn17tensor}.

\subsection{Reminder on the complete cobar construction}

What we explain here is a special case of \cite[Ch. 11.1--3]{vallette12}, namely where we take $\P = \lie$ and only consider the canonical twisting morphism $\pi:\bar\lie\to\lie$, where $\bar\lie$ is the bar construction of the operad $\lie$ encoding Lie algebras. In fact, we consider a slight variation on the material presented there, as we remove the conilpotency condition on coalgebras but additionally add the requirement that algebras be complete. See also \cite[Sect. 6.2]{rn17tensor}.

\medskip

Let $X$ be a dg $\bar\lie$-coalgebra. The \emph{complete cobar construction} of $X$ is the complete dg Lie algebra
\[
\widehat{\Omega}_\pi X\coloneqq\left(\widehat{\lie}(X),d\coloneqq d_1 + d_2\right),
\]
where
\[
\widehat{\lie}(X)\coloneqq\prod_{n\ge1}\lie(n)\otimes_{\S_n}X^{\otimes n}
\]
and where the differential $d$ is composed by the following two parts:
\begin{enumerate}
	\item The differential $-d_1$ is the unique derivation extending the differential $d_X$ of $X$.
	\item The differential $-d_2$ is the unique derivation extending the composite
	\[
	X\xrightarrow{\Delta_X}\widehat{\bar\lie}(X)\xrightarrow{\pi\circ 1_X}\widehat{\lie}(X)\ .
	\]
	Notice that as $X$ is not assumed to be conilpotent, the decomposition map $\Delta_X$ really lands in the product
	\[
	\widehat{\bar\lie}(X)\coloneqq \prod_{n\ge0}\left(\bar\lie(n)\otimes X^{\otimes n}\right)^{\S_n}
	\]
	and not the direct sum. Thus it is necessary to consider the free complete Lie algebra over $X$. Also, there is a passage from invariants to coinvariants that is left implicit here, as the decomposition map lands in invariants, but the elements of the complete free Lie algebra $\widehat{\lie}(X)$ are coinvariants. This introduces factors of the form $\tfrac{1}{n!}$ when computing explicit formul{\ae} for $d_2$.
\end{enumerate}
The complete cobar construction $\widehat{\Omega}_\pi$ defines a functor from dg $\bar\lie$-coalgebras to complete dg Lie algebras.

\subsection{Representing $\MC(\g\otimes C_\bullet)$}

Using the Dupont contraction, the Homotopy Transfer Theorem gives the structure of a simplicial $\C_\infty$-algebra to $C_\bullet$. As the underlying cochain complex $C_n$ is finite dimensional for each $n$, it follows that its dual is a cosimplicial $\bar(\susp\otimes\lie)$-coalgebra. Therefore, the desuspension $sC_\bullet^\vee$ is a cosimplicial $\bar\lie$-coalgebra, and we can take its complete cobar construction.

\begin{definition}
	We denote this cosimplicial dg Lie algebra by $\mc_\bullet\coloneqq\widehat{\Omega}_\pi(sC_\bullet^\vee)$.
\end{definition}

\begin{theorem} \label{thm:cosimplicialModel}
	Let $\g$ be a complete dg Lie algebra. There is a canonical isomorphism
	\[
	\MC(\g\otimes C_\bullet)\cong\hom_{\mathsf{dgLie}}(\mc_\bullet,\g)\ .
	\]
	It is natural in $\g$.
\end{theorem}

\begin{proof}
	By \cite[Th. 5.1]{rn17tensor}, the $\L_\infty$-algebra structure we have on $\g\otimes C_\bullet$ is the same as the structure that we obtain on the tensor product of the dg Lie algebra $\g$ with the simplicial $\C_\infty$-algebra $C_\bullet$ by using \cite[Th. 3.4]{rn17tensor} with $\P=\Q=\lie$ and $\Psi=\id_\lie$. Therefore, we can apply \cite[Cor. 6.6]{rn17tensor}, which gives the desired isomorphism.
\end{proof}

With this form for $\MC(\g\otimes C_\bullet)$, Theorem \ref{thm:mainThm} reads as follows.

\begin{corollary} \label{cor:exMainThm}
	Let $\g$ be a complete dg Lie algebra. There is a weak equivalence of simplicial sets
	\[
	\MC_\bullet(\g)\simeq\hom_{\mathsf{dgLie}}(\mc_\bullet,\g)\ ,
	\]
	natural in $\g$.
\end{corollary}

We can completely characterize the first levels of the cosimplicial dg Lie algebra $\mc_\bullet$. Recall from  \cite{lawrence10} the Lawrence--Sullivan algebra: it is the unique free complete dg Lie algebra generated by two Maurer--Cartan elements in degree $1$ and a single element in degree $0$ such that the element in degree $0$ is a gauge between the two generating Maurer--Cartan elements.

\begin{proposition}
	The first two levels of the cosimplicial dg Lie algebra $\mc_\bullet$ are as follows.
	\begin{enumerate}
		\item \label{n1} The dg Lie algebra $\mc_0$ is isomorphic to the free dg Lie algebra with a single Maurer--Cartan element as the only generator.
		\item \label{n2} The dg Lie algebra $\mc_1$ is isomorphic to the Lawrence--Sullivan algebra.
	\end{enumerate}
\end{proposition}

\begin{proof}
	For (\ref{n1}), we have $\Omega_0\cong\k\cong C_0$, both $p_0$ and $i_0$ are the identity, and $h_0 = 0$. It follows that, as a complete graded free Lie algebra, $\mc_0$ is given by
	\[
	\mc_0 = \widehat{\lie}(s\k)\ .
	\]
	We denote the generator by $\alpha\coloneqq s1^\vee$. It has degree $1$. Let $\g$ be any complete dg Lie algebra, then a morphism
	\[
	\phi:\mc_0\longrightarrow\g
	\]
	is equivalent to the Maurer--Cartan element
	\[
	\phi(\alpha)\otimes 1\in\MC(\g\otimes C_\bullet) \cong \MC(\g)\ .
	\]
	Conversely, through $P_0$ every Maurer--Cartan element of $\g$ induces a morphism $\mc_0\to\g$. As this is true for any dg Lie algebra $\g$, it follows that $\alpha$ is a Maurer--Cartan element.
	
	\medskip
	
	To prove (\ref{n2}), we start by noticing that
	\[
	C_1\coloneqq \k\omega_0\oplus\k\omega_1\oplus\k\omega_{01}
	\]
	with $\omega_0,\omega_1$ of degree $0$ and $\omega_{01}$ of degree $1$. Denoting by $\alpha_i\coloneqq s\omega_i^\vee$ and by $\lambda\coloneqq s\omega_{01}^\vee$, we have
	\[
	\mc_1 = \widehat{\lie}(\alpha_0,\alpha_1,\lambda)
	\]
	as a graded Lie algebra. Let $\g$ be any dg Lie algebra, then a morphism
	\[
	\phi:\mc_1\longrightarrow\g
	\]
	is equivalent to a Maurer--Cartan element
	\[
	\phi(\alpha_0)\otimes\omega_0 + \phi(\alpha_1)\otimes\omega_1 + \phi(\lambda)\otimes\omega_{01}\in\MC(\g\otimes C_1)\ ,
	\]
	see \cite[Sect. 6.3--4]{rn17tensor}. Applying $I_1$, as in the proof of Proposition \ref{prop:forElModel} we obtain
	\[
	I_1(\phi(\alpha_0)\otimes\omega_0 + \phi(\alpha_1)\otimes\omega_1 + \phi(\lambda)\otimes\omega_{01}) = a(t_0,t_1) + \phi(\lambda)\otimes\omega_{01}\in\MC_1(\g)
	\]
	with $a(1,0) = \phi(\alpha_0)$ and $a(0,1) = \phi(\alpha_1)$. The Maurer--Cartan equation for $a(t_0,t_1) + \phi(\lambda)\otimes\omega_{01}$ then shows that $\phi(\lambda)$ is a gauge from $\phi(\alpha_0)$ to $\phi(\alpha_1)$. Conversely, if we are given the data of two Maurer--Cartan elements of $\g$ and a gauge equivalence between them, then this data gives us a Maurer--Cartan element of $\g\otimes\Omega_1$. Applying $P_1$ then gives back a non-trivial morphism $\mc_1\to\g$. As this is true for any $\g$, it follows that $\mc_1$ is isomorphic to the Lawrence--Sullivan algebra.
\end{proof}

\begin{remark}
	Alternatively, one could write down explicitly the differentials for both $\mc_0$ (which is straightforward) and $\mc_1$ (with the help of \cite[Prop. 19]{cheng08}). An explicit description of $\mc_\bullet$ is made difficult by the fact that one needs to know the whole $\C_\infty$-algebra structure on $C_\bullet$ in order to write down a formula for the differential.
\end{remark}

\subsection{Relations to rational homotopy theory}

The cosimplicial dg Lie algebra $\mc_\bullet$ has already made its appearance in the literature not long ago, in the paper \cite{buijs15}, in the context of rational homotopy theory, where it plays the role of a Lie model for the geometric $n$-simplex. With the goal of simplifying comparison and interaction between our work and theirs, we provide here a short review and a dictionary between our vocabulary and the notation used in \emph{op. cit.}.

\medskip

\begin{center}
	\begin{tabular}{|l|l|}
	\hline
	Notation of this paper & Notation of \cite{buijs15}\\
	\hline
	$\mc_\bullet$ & $\mathfrak{L}_\bullet$ or $\mathfrak{L}_{\Delta^\bullet}$\\
	$\Omega_\bullet$ & $A_{PL}(\Delta^\bullet)$\\
	$\bar_\iota$ & Quillen functor $\mathcal{C}$\\
	$\hom_{\mathsf{dgLie}}(\mc_\bullet,-)$ & $\langle-\rangle$\\
	$\hom_{\mathsf{dgCom}}(-,\Omega_\bullet)$ & $\langle -\rangle_S$\\
	\hline
	\end{tabular}
\end{center}

\medskip

\begin{remark}
	The fact that the cosimplicial dg Lie algebra $\mc_\bullet$ is isomorphic to $\mathfrak{L}_\bullet$ is immediate from \cite[Def. 2.1 and Thm. 2.8]{buijs15}.
\end{remark}

The following theorem has non-empty intersection with our results. We say a dg Lie algebra is of \emph{finite type} if it is finite dimensional in every degree and if its degrees are bounded either above or below.

\begin{theorem}[{\cite[Th. 8.1]{buijs15}}]
	Let $\g$ be a dg Lie algebra of finite type with $H^n(\g,d) = 0$ for all $n>0$. Then there is a homotopy equivalence of simplicial sets
	\[
	\hom_{\mathsf{dgLie}}(\mc_\bullet,\g)\simeq\hom_{\mathsf{dgCom}}(\bar_\iota(s\g)^\vee,\Omega_\bullet)\ .
	\]
\end{theorem}

We can easily recover an analogous result, which works on complete dg Lie algebras of finite type such that $\g^{-1}=0$, but without restrictions on the cohomology, using our main theorem and some results of \cite{rn17tensor}.

\begin{proposition}
	Let $\g$ be a complete dg Lie algebra of finite type such that $\g^{-1}=0$. Then there is a weak equivalence of simplicial sets
	\[
	\hom_{\mathsf{dgLie}}(\mc_\bullet,\g)\simeq\hom_{\mathsf{dgCom}}(\bar_\iota(s\g)^\vee,\Omega_\bullet)\ .
	\]
\end{proposition}

\begin{proof}
	The proof is given by the sequence of equivalences
	\begin{align*}
		\hom_{\mathsf{dgCom}}(\bar_\iota(s\g)^\vee,\Omega_\bullet)\cong&\ \hom_{\mathsf{dgCom}}\left(\widehat{\Omega}_\pi(s^{-1}\g^\vee),\Omega_\bullet\right)\\
		\cong&\ \MC(\g\otimes\Omega_\bullet)\\
		\simeq&\ \hom_{\mathsf{dgLie}}(\mc_\bullet,\g).
	\end{align*}
	In the first line we used the natural isomorphism
	\[
	\bar_\iota(s\g)^\vee\cong\widehat{\Omega}_\pi(s^{-1}\g^\vee)\ .
	\]
	Notice that the assumptions on $\g$ make it so that $\g^\vee$ is a $\lie^\vee$-coalgebra. In the second line we used a slight generalization of \cite[Cor. 6.6]{rn17tensor} for $\Q=\P=\com$ and $\Psi$ the identity morphism of $\com$. Notice that here the assumption that $\g^{-1}=0$ makes it so that
	\[
	\hom_{\mathsf{dgCom}}\left(\widehat{\Omega}_\pi(s^{-1}\g^\vee),\Omega_\bullet\right)\cong\hom(s^{-1}\g^\vee,\Omega_\bullet)^0
	\]
	even though $\Omega_\bullet$ is not complete. Finally, in the third line we used our Corollary \ref{cor:exMainThm}.
\end{proof}

\bibliographystyle{alpha}
\bibliography{Representing_the_Deligne-Hinich-Getzler_infinity-groupoid}

\end{document}